\documentclass[twoside,leqno,10pt, A4]{amsart}
\usepackage{amsfonts}
\usepackage{amsmath}
\usepackage{amscd}
\usepackage{amssymb}
\usepackage{amsthm}
\usepackage{amsrefs}
\usepackage{latexsym}
\usepackage{mathrsfs}
\usepackage{bbm}
\usepackage{amscd}
\usepackage{amssymb}
\usepackage{amsthm}
\usepackage{amsrefs}
\usepackage{latexsym}
\usepackage{mathrsfs}
\usepackage{bbm}
\usepackage{enumerate}
\usepackage{graphicx}
\usepackage{color}
\begin{document}

\newtheorem{theorem}[subsection]{Theorem}
\newtheorem{proposition}[subsection]{Proposition}
\newtheorem{lemma}[subsection]{Lemma}
\newtheorem{corollary}[subsection]{Corollary}
\newtheorem{conjecture}[subsection]{Conjecture}
\newtheorem{prop}[subsection]{Proposition}
\newtheorem{defin}[subsection]{Definition}

\numberwithin{equation}{section}
\newcommand{\mr}{\ensuremath{\mathbb R}}
\newcommand{\mc}{\ensuremath{\mathbb C}}
\newcommand{\dif}{\mathrm{d}}
\newcommand{\intz}{\mathbb{Z}}
\newcommand{\ratq}{\mathbb{Q}}
\newcommand{\natn}{\mathbb{N}}
\newcommand{\comc}{\mathbb{C}}
\newcommand{\rear}{\mathbb{R}}
\newcommand{\prip}{\mathbb{P}}
\newcommand{\uph}{\mathbb{H}}
\newcommand{\fief}{\mathbb{F}}
\newcommand{\majorarc}{\mathfrak{M}}
\newcommand{\minorarc}{\mathfrak{m}}
\newcommand{\sings}{\mathfrak{S}}
\newcommand{\fA}{\ensuremath{\mathfrak A}}
\newcommand{\mn}{\ensuremath{\mathbb N}}
\newcommand{\mq}{\ensuremath{\mathbb Q}}
\newcommand{\half}{\tfrac{1}{2}}
\newcommand{\f}{f\times \chi}
\newcommand{\summ}{\mathop{{\sum}^{\star}}}
\newcommand{\chiq}{\chi \bmod q}
\newcommand{\chidb}{\chi \bmod db}
\newcommand{\chid}{\chi \bmod d}
\newcommand{\sym}{\text{sym}^2}
\newcommand{\hhalf}{\tfrac{1}{2}}
\newcommand{\sumstar}{\sideset{}{^*}\sum}
\newcommand{\sumprime}{\sideset{}{'}\sum}
\newcommand{\sumprimeprime}{\sideset{}{''}\sum}
\newcommand{\sumflat}{\sideset{}{^\flat}\sum}
\newcommand{\shortmod}{\ensuremath{\negthickspace \negthickspace \negthickspace \pmod}}
\newcommand{\V}{V\left(\frac{nm}{q^2}\right)}
\newcommand{\sumi}{\mathop{{\sum}^{\dagger}}}
\newcommand{\mz}{\ensuremath{\mathbb Z}}
\newcommand{\leg}[2]{\left(\frac{#1}{#2}\right)}
\newcommand{\muK}{\mu_{\omega}}
\newcommand{\thalf}{\tfrac12}
\newcommand{\lp}{\left(}
\newcommand{\rp}{\right)}
\newcommand{\Lam}{\Lambda_{[i]}}
\newcommand{\lam}{\lambda}
\newcommand{\af}{\mathfrak{a}}
\newcommand{\sw}{S_{[i]}(X,Y;\Phi,\Psi)}
\newcommand{\lz}{\left(}
\newcommand{\pz}{\right)}
\newcommand{\bfrac}[2]{\lz\frac{#1}{#2}\pz}
\newcommand{\odd}{\mathrm{\ primary}}
\newcommand{\even}{\text{ even}}
\newcommand{\res}{\mathrm{Res}}

\theoremstyle{plain}
\newtheorem{conj}{Conjecture}
\newtheorem{remark}[subsection]{Remark}

\makeatletter
\def\widebreve{\mathpalette\wide@breve}
\def\wide@breve#1#2{\sbox\z@{$#1#2$}%
     \mathop{\vbox{\m@th\ialign{##\crcr
\kern0.08em\brevefill#1{0.8\wd\z@}\crcr\noalign{\nointerlineskip}%
                    $\hss#1#2\hss$\crcr}}}\limits}
\def\brevefill#1#2{$\m@th\sbox\tw@{$#1($}%
  \hss\resizebox{#2}{\wd\tw@}{\rotatebox[origin=c]{90}{\upshape(}}\hss$}
\makeatletter

\title[First moment of central values of primitive Dirichlet $L$-functions with fixed order characters]{First moment of central values of some primitive Dirichlet $L$-functions with fixed order characters}


\author[P. Gao]{Peng Gao}
\address{School of Mathematical Sciences, Beihang University, Beijing 100191, China}
\email{penggao@buaa.edu.cn}

\author[L. Zhao]{Liangyi Zhao}
\address{School of Mathematics and Statistics, University of New South Wales, Sydney, NSW 2052, Australia}
\email{l.zhao@unsw.edu.au}

\begin{abstract}
 We evaluate asymptotically the smoothed first moment of central values of families of primitive cubic, quartic and sextic Dirichlet $L$-functions, using the method of double Dirichlet series.  Quantitative non-vanishing result for these $L$-values are also proved.
\end{abstract}

\maketitle



\section{Introduction}
\label{sec 1}

   Moments of central values of $L$-functions attached to characters of a fixed order have attracted increasing attention in the literature as they have many interesting arithmetic applications. Although there are many results concerning quadratic Dirichlet $L$-functions (see \cites{Jutila, DoHo, MPY, ViTa, sound1, DGH, Young1}), fewer results are available for those of higher orders.  In \cite{B&Y}, S. Baier and M. P. Young evaluated the first moment of the family of cubic Dirichlet $L$-functions at the central point. Their result was extended to the quartic case in \cite{G&Zhao7} assuming the truth of the Lindel\"of hypothesis. \newline

  While the {\it modus operandi} of establishing asymptotic formulas for the first moment in \cites{B&Y, G&Zhao7} utilizes classical tools such as approximate functional equations, B. Brubaker \cite{Brubaker} applied the powerful tool of multiple Dirichlet series to study the second moment of central values of the cubic family of Dirichlet $L$-functions. To achieve the result, certain correction factors are attached to the triple Dirichlet series under consideration in \cite{Brubaker} in order to develop sufficient functional equations to obtain analytical continuation of the series to the entire $\mc^3$. \newline

 It is the aim of this paper to investigate the first moments of central values of the families of cubic, quartic and sextic Dirichlet $L$-functions  using the method of multiple Dirichlet series. Our treatment differs from that in \cite{Brubaker} as we do not seek to establish analytical continuation of the double Dirichlet series involved in our proof to the whole of $\mc^2$, but rather to a region that is large enough for our purpose. The advantage of doing so is that altering the underlying series becomes necessary so that our approach is more concise and direct. \newline

Let $\Phi(x)$ be a fixed non-negative, smooth function compactly supported on the set of positive real numbers ${\mr}_+$. We also write $\widehat f$ for the Mellin transform of a function $f$. Further let  $\chi_0$ be the principal Dirichlet character. Our main result is as follows.
\begin{theorem}
\label{firstmoment}
  With the notation as above and assuming the truth of the generalized Lindel\"of hypothesis, let $j=3, 4$ or $6$.  We have, for $1/2>\Re(\alpha) \geq 0$, all large $Q$ and any $\varepsilon>0$,
\begin{equation}
\label{eq:1}
 \sum_{(q,j)=1}\;  \sumstar_{\substack{\chi \bmod{q} \\ \chi^j = \chi_0}} L(\half+\alpha, \chi) \Phi \leg{q}{Q} = C_j Q \widehat{\Phi}(0) + O((2+|\alpha|)^{3(j-1)+\varepsilon}Q^{\frac {2j+1-2j\alpha}{2j+2}+ \varepsilon}),
\end{equation}
where $C_j$ is a positive constant explicitly given in \eqref{eq:c} and $\sum^*$ over $\chi$ means the sum runs over primitive characters $\chi$ such that $\chi^i, 1 \leq i \leq j-1$ remains primitive.
\end{theorem}

  As mentioned earlier, the proof of Theorem \ref{firstmoment} uses the method of multiple Dirichlet series. Moreover, we make use of Lemma \ref{lemma:quarticclass} to relate Dirichlet characters of order $3,4$ and $6$ to Hecke characters of corresponding orders in quadratic number fields $K=\mq(i)$ or $\mq(\sqrt{-3})$. The work of S. J. Patterson \cite{P} on Gauss sums associated to higher order Hecke characters also plays an important role. \newline

  We also point out here that we need to assume the truth of the Lindel\"of hypothesis in our proof to control the size of certain $L$-values on average. In fact, this is only used to control the size of the right-hand side of \eqref{Aswboundsumm}. One may alternatively estimate the same expression using \cite[(39)]{B&Y} by noting that \cite[(39)]{B&Y} continues to hold with $1/2+it$ there being replaced by any $s$ with $1/2 <\Re(s) <1$. This and partial summation then leads to an error term in \eqref{eq:1} of size $O(Q^{37/38 + \varepsilon})$ unconditionally, which matches the result given in \cite[Theorem 1.1]{B&Y}.  However, as this is inferior to the error term $O(Q^{13/14 + \varepsilon})$ given in an earlier version ({\tt arXiv:0804.2233v1}) of \cite{B&Y}, we shall not get into the details here. \newline

  Similar to the proof of \cite[Corollary 1.2]{B&Y}, the well-known bound for the eighth moment of Dirichlet $L$-functions and the lower bound implied by the asymptotic formula in Theorem~\ref{firstmoment} yield, via H\"older's inequality, the following non-vanishing result on the central values of Dirichlet $L$-functions under our consideration.
\begin{corollary} \label{coro:nonvanish}
Assume that the generalized Lindel\"of hypothesis is true.  There exist infinitely many primitive Dirichlet characters $\chi$ of order $j=3,4$ or $6$ such that $L(1/2, \chi) \neq 0$.  More precisely, the number of such characters with conductor $\leq Q$ is $\gg Q^{6/7- \varepsilon}$.
\end{corollary}

We note that the cases $j=3$ and $4$, Corollary~\ref{coro:nonvanish} is already contained in \cite[Corollary 1.2]{B&Y} and  \cite[Corollary 1.2]{G&Zhao7}, respectively.  Throughout the paper, $\varepsilon$, as usual, represents a small positive real number which may not be the same in each occurrence.

\section{Preliminaries}
\label{sec 2}

\subsection{Residue symbols}
\label{sec2.4}

   For any number field $K$, let $\mathcal{O}_K, U_K, D_K$ and $K^{\times}$ denote the ring of integers, the group of units, the discriminant of $K$ and the multiplicative group $K \setminus \{0 \}$, respectively. In this paper, we fix $K=\mq(i)$ or $\mq(\sqrt{-3})$. Then the ring of integers $\mathcal{O}_K$ is a free $\mz$ module such that $\mathcal{O}_K=\mz+i \mz$ (resp. $\mz+\omega \mz$) when $K=\mq(i)$ (resp. $\mq(\sqrt{-3})$). Here $\omega=(-1+\sqrt{-3})/2$.  The group of units $U_K= \{ \pm 1 , \pm i \}, D_K=-4$ (resp. $U_K=\{ \pm 1, \pm \omega, \pm \omega^2 \}, D_K=-3$) when $K=\mq(i)$ (resp. $\mq(\sqrt{-3})$).  Let $\delta_K=\sqrt{D_K}$ so that the different of $K$ is the principal ideal $(\delta_K)$. These facts can be found in \cite[Section 3.8]{iwakow}. \newline

  We write $\chi^{(m)}_j$ (resp. $\chi_{j, m}$) for the $j$-order residue symbol $\leg {m}{\cdot}_j$ (resp.  $\leg {\cdot}{m}_j$) in a number field defined for $(m,j)=1$ in \cite[Section 4.1]{Lemmermeyer}. Note that the existence of these symbols depends on whether the corresponding number field contains a primitive $j$-th root of unity. In particular, this implies that the $j$-order residue symbols for $j=4$ (resp. $j=3,6$) are defined when $K=\mq(i)$ (resp. $K=\mq(\sqrt{-3})$). \newline

 For any positive rational integer $n$, we say an element $c \in \mathcal{O}_K$ is $n$-th power free if no $n$-th power of any prime divides $c$. We note the following classification of all primitive Dirichlet characters of order $3,4$ and $6$.  We omit its proof here as it is similar to \cite[Lemma 2.1]{B&Y}.
\begin{lemma}
\label{lemma:quarticclass}
 For $j=3,4,6$, the primitive $j$-th order Dirichlet characters of conductor $q$ co-prime to $j$ such that their $i$-th power remain primitive for $1 \leq i \leq j-1$ are of the form $\widehat \chi_{j,n}:m \mapsto \leg{m}{n}_j$ for some primary, square-free $n \in \mathcal O_K$ such that $N(n) = q$ and $n$ is not divisible by any rational primes.
\end{lemma}

Let $N(q)$ stand for the norm of any $q \in \mathcal O_K$. Our next result evaluates $\chi^{(m)}_j(d)$ for at rational integers $m, d$.
\begin{lemma}
\label{lemma:chivalue}
 Let $j \in \{ 3,4, 6\}$ and $m, d \in \mz$ with $(md, j)=1$ and $(m,d)=1$. Then we have
\begin{align}
\label{chivalue}
\begin{split}
 \chi^{(m)}_j(d)=1.
\end{split}
\end{align}
\end{lemma}
\begin{proof}
 The proof for the cases of $j=3$ and $4$ can be found in the Corollary to Proposition 9.3.4 in \cite{I&R} and \cite[Proposition 9.8.4]{I&R}, respectively. Next note that \cite[Proposition 4.2, iii)]{Lemmermeyer} implies that
\begin{align*}
\begin{split}
 \leg {m}{d}_2=\leg {N(m)}{d}_{\mz}=\leg {m^2}{d}_{\mz}=1,
\end{split}
\end{align*}
  where $\leg {\cdot}{\cdot}_{\mz}$ is the Kronecker symbol in $\mz$. It follows that
\begin{align*}
\begin{split}
 \leg {m}{d}_6=\leg {m}{d}^3_6\leg {m}{d}^{-2}_6=\leg {m}{d}_{2}=\leg {m}{d}^{-1}_{3}=1.
\end{split}
\end{align*}
  This completes the proof.
\end{proof}

    It is well-known that every ideal of $\mathcal O_K$ is principal for $K=\mq(i)$ or $\mq(\sqrt{-3})$, so that one may fix
a unique generator for each non-zero ideal. We now describe a set of generators for each ideal co-prime to a given ideal and we call these generators primary elements. For $K=\mq(i)$, we say an element $n$ is primary if $(n, 2)=1$ and $n \equiv 1 \pmod {(1+i)^3}$. Note that $(1+i)$ is the only ideal above the rational ideal $(2) \in \mq$ and that the group $\left (\mathcal{O}_K / ((1+i)^3) \right )^{\times}$ is isomorphic to $U_K$. It follows that every ideal co-prime to $(2)$ in $\mathcal O_K$ has a unique primary generator. \newline

 For $K=\mq(\sqrt{-3})$, we define two sets of primary elements depending on whether a cubic residue symbol or a sextic residue symbol is involved. To distinguish them, we shall call the elements in the first set by primary elements and the elements in the second set by $E$-primary elements. We first note that $(1-\omega)$ is the only ideal above the rational ideal $(3) \in \mq$ (see \cite[Proposition 9.1.4]{I&R}). We then
define for any $n \in \mathcal O_K$ to be primary if $(n, 3)=1$ and $n \equiv 1 \pmod {3}$.  As the group $\left (\mathcal{O}_K / (3) \right )^{\times}$ is isomorphic to $U_K$, it follows that every ideal co-prime to $(3)$ in $\mathcal O_K$ has a unique primary generator.  Next, we note that the rational ideal $(2)$ is inert in $\mathcal O_K$ and we define $n=a+b\omega \in \mathcal O_K$ with $a, b \in
\mz$  to be $E$-primary if $(n,6)=1$, $n \equiv \pm 1 \pmod 3$ and $n$ satisfies
\begin{align*}
\begin{split}
   & a+b \equiv 1 \pmod 4, \quad  \text{if} \quad 2 | b,  \\
   & b \equiv 1 \pmod 4, \quad  \text{if} \quad 2 | a,  \\
   & a \equiv 3 \pmod 4, \quad  \text{if} \quad 2 \nmid ab.
\end{split}
\end{align*}
  Our definition of $E$-primary elements follows from  the notations in \cite[Section 7.3]{Lemmermeyer}. It is shown in \cite[Lemma
7.9]{Lemmermeyer} that for any $(n, 6)=1$, the following two statements are equivalent.
\begin{enumerate}
\item $n$ is $E$-primary.
\item$n^3=c+d\omega$ with $c, d \in \mz$ such that $6 | d$ and $c+d \equiv 1 \pmod{4}$.
\end{enumerate}
  Thus products of $E$-primary elements are again $E$-primary. It is easy to see that each ideal co-prime to $6$ in $\mathcal O_K$ has a unique primary or $E$-primary generator. \newline

 For $K=\mq(i)$, the following quartic reciprocity law (see \cite[Theorem 6.9]{Lemmermeyer}) holds for two co-prime primary elements $m, n
\in \mathcal{O}_{K}$ with $(mn, 2)=1$,
\begin{align}
\label{quartrec}
 \leg{m}{n}_4 = \leg{n}{m}_4(-1)^{((N(n)-1)/4)((N(m)-1)/4)}.
\end{align}

   For $K=\mq(\sqrt{-3})$, we have the following cubic reciprocity law (see \cite[Theorem 7.8]{Lemmermeyer}) for two co-prime primary
elements $m, n \in \mathcal{O}_{K}$ with $(mn, 3)=1$,
\begin{align}
\label{cubicrec}
    \leg {n}{m}_3 =\leg{m}{n}_3.
\end{align}
Moreover, the following sextic reciprocity law holds for two $E$-primary, co-prime numbers $n, m \in \mathcal O_K$ with $(mn, 6)=1$,
\begin{align}
\label{quadreciQw}
    \leg {n}{m}_6 =\leg{m}{n}_6(-1)^{((N(n)-1)/2)((N(m)-1)/2)}.
\end{align}

\subsection{Gauss sums}
\label{section:Gauss}

  Let $K$ be a number field of class number one. Following the nomenclature of \cite[Section 3.8]{iwakow}, we say a Dirichlet character $\chi$ modulo $(q) \neq (0)$ is a homomorphism
\begin{align*}
  \chi: \left (\mathcal{O}_K / (q) \right )^{\times}  \rightarrow S^1 :=\{ z \in \mc :  |z|=1 \}.
\end{align*}
Also, $\chi$ is said to be primitive modulo $(q)$ if it does not factor through $\left (\mathcal{O}_K / (q') \right )^{\times}$ for any divisor $q'$ of $q$ with $N(q')<N(q)$. \newline

  If further $\chi(u)=1$ for any $u \in U_K$, we may also regard $\chi$ as defined on ideals of $\mathcal O_K$ since every ideal in $\mathcal O_K$ is principal. We then say such a character $\chi$ is a Hecke character modulo $(q)$ of trivial infinite type. In this case, the Hecke character $\chi$ is primitive if it is primitive as a Dirichlet character. We say that $\chi$ is a Hecke character modulo $q$ instead of modulo $(q)$ when there is no ambiguity. \newline

   For any Dirichlet character $\chi$ modulo $(q)$ and any $k \in \mathcal O_K$, we define the associated Gauss sum $g_K(k, \chi; q)$ by
\begin{align*}
 g_K(k,\chi; q ) = \sum_{x \shortmod{q}} \chi(x) \widetilde{e}_K\leg{kx}{q}, \quad \mbox{where} \quad    \widetilde{e}_K(z) =\exp \left( 2\pi i  \left( \frac {z}{\sqrt{D_K}} -
\frac {\overline{z}}{\sqrt{D_K}} \right) \right).
\end{align*}
We note that the definition of $g_K(k,\chi; q)$ depends on the choice of the generator $q$ of the ideal $(q)$. But if $\chi$ is a Hecke character modulo $(q)$ of trivial infinite type, it is easy to see that $g_K(k,\chi; q)$ is independent of the choice of such a generator.  In this case, we write $g_K(k,\chi)$ for $g_K(k,\chi; q)$ and $g_K(\chi)$ for $g_K(1,\chi)$. Notice that whenever $\chi_{j, n}$ it is defined, it is a Dirichlet character modulo $(n)$. We then write $g_K(k, \chi_{j, n})$ for $g_K(k, \chi_{j, n}; n)$.  Moreover, for the special case when $K=\mq$, we always take the generator $q$ for an ideal $0 \neq (q)\in \mz$ with $q>0$. We then write $\tau(k, \chi)$ for $g_\mq(k, \chi)$ and $\tau(\chi)$ for $\tau(1, \chi)$. Note that $D_\mq=1$ so that it follows from the definition that for any $h \in \mz$,
\begin{equation*}
  \tau(h, \chi) =  \sum_{1 \leq x \leq q}\chi(x) e \left( \frac{hx}{q} \right).
\end{equation*}
Here $e(z) = \exp (2 \pi i z)$ for any complex number $z$. \newline

  We have the following properties for $g_K(k,\chi_{j, n})$.
\begin{lemma}
   Let $j=3,4$ or $6$. We have
\begin{align}
\label{eq:gmult}
 g_K(rs, \chi_{j, n})  =& \overline{\leg{s}{n}}_j g_{K}(r,\chi_{j, n}), \quad (s,n)=1, \\
\label{2.03}
   g_{K}(r,\chi_{j, n_1n_2}) =& \leg{n_2}{n_1}_j\leg{n_1}{n_2}_jg_{K}(r, \chi_{j, n_1}) g_{K}(r, \chi_{j, n_1}), \quad (n_1, n_2) = 1, \\
\label{grel}
 g_K(\chi_{j,dn}) =& \overline \chi_{j,n}(d^{j-2})g_K(\chi_{j,d})g_K(\chi_{j,n}), \quad (n, d) = 1, \ d \in \mz, \ n, d \odd , \\
\label{2.1}
   \Big|g_K(\chi_{j,n})\Big | =& \begin{cases}
    \sqrt{N(n)} \qquad & \text{if $n$ is square-free}, \\
     0 \qquad & \text{otherwise}.
    \end{cases}
\end{align}
\end{lemma}
\begin{proof}
  The identity \eqref{eq:gmult} follows easily from the definition and similarly \eqref{2.03} emerges upon using the Chinese remainder theorem to write $x \pmod n$ as $x_1n_1+x_2n_2$ with $x_i$ varying $\pmod {n_i}, \ i=1,2$. For \eqref{grel}, observe that a primary rational integer $d$ must be odd when $j=4$ or $6$ and $N(d)=d^2 \equiv 1 \pmod 4$ for any odd rational integer $d \in \mz$.  This together with various reciprocity laws \eqref{quartrec} and \eqref{quadreciQw} then implies that for primary elements $n, d \in \mathcal O_K$ with $(n, d)=1$, we have for $j=4,6$,
\begin{align*}
   \leg{n}{d}_j=\leg{d}{n}_j.
\end{align*}
  On the other hand, when $j=3$, we derive from \eqref{cubicrec} and the multiplicity of cubic symbol that the above continues to hold when $2|d$ in this case. \newline

 We then deduce that for these $n$, $d$ and for all $j$'s under our consideration,
\begin{align*}
\begin{split}
g_K(\chi_{j,dn}) = \leg{d}{n}_j\leg{n}{d}_jg_K(\chi_{j,d})g_K(\chi_{j,n})=\leg{d}{n}^2_jg_K(\chi_{j,d})g_K(\chi_{j,n}).
\end{split}
\end{align*}
 The above implies \eqref{grel}. Lastly, the relation \eqref{2.1} is given on \cite[p. 195]{P}). This completes the proof of the lemma.
\end{proof}

     Recall from Lemma \ref{lemma:quarticclass} that $\widehat \chi_{j,n}$ is a primitive $j$-th order Dirichlet character for a primary $n \in \mathcal O_K$. We then have the following properties concerning $\tau$.
\begin{lemma}
\label{quarticGausssum1}
   With the notation as above. We have for $j=3,4$ or $6$,
\begin{align}
\label{tauprim}
  \tau(h, \chi) =&  \overline{\chi}(h)\tau(\chi), \quad \text{if $\chi$ is primitive}, \\
\label{tauprim1}
   \tau(\widehat\chi_{j,n})  =& \begin{cases}
      \overline{\leg {\sqrt{D_K}}{n}}_jg_K(\chi_{j,n}), \qquad & j=3, \\
     i^{(1-\chi_{4,n}(-1))/2}\overline{\leg {(2i)^3}{n}}_j g_K(\chi_{j,n}), \qquad & j=4, \\
     \overline{\leg {-D^2_K}{n}}_jg_K(\chi_{j,n}), \qquad & j=6.
    \end{cases}
\end{align}
\end{lemma}
\begin{proof}
  The proof of \eqref{tauprim} follows from \cite[\S 9]{Da}.  To establish \eqref{tauprim1}, using arguments similar to those given on \cite[p. 884-885]{B&Y} (be aware of the difference between the definition of $g(n)$ in \cite{B&Y} and our definition of $ g_K(\chi_{j,n})$), we get
\begin{equation}
\label{tau}
 \tau(\widehat\chi_{j,n}) =   \overline{\leg {\sqrt{D_K}}{n}}_j\leg {\overline{n}}{n}_jg_K(\chi_{j,n}).
\end{equation}

  Note that we have $\leg {\overline{n}}{n}_3=1$ for $n \equiv \pm 1 \pmod 3$ (see \cite[p. 884--885]{B&Y}). This gives the formula for $j=3$ in \eqref{tauprim1}. It also follows that we have
\begin{equation}
\label{n6simplified}
  \leg {\overline{n}}{n}_6=\leg {\overline{n}}{n}^3_6\leg {\overline{n}}{n}^{-2}_6=\leg {\overline{n}}{n}_2\leg {\overline{n}}{n}^{-1}_3=\leg {\overline{n}}{n}_2.
\end{equation}
 It is shown on \cite[p. 367--368]{G&Zhao3} that
\begin{align}
\label{n2}
  \leg {\sqrt{D_K}\overline{n}}{n}_2=\leg {-1}{n}_2.
\end{align}

  We deduce readily from \eqref{tau}--\eqref{n2} that
\begin{align*}
 \tau(\widehat\chi_{6,n}) =   \overline{\leg {\sqrt{D_K}}{n}}_6 \overline{\leg {\sqrt{D_K}}{n}}_2 \leg {-1}{n}_2g_K(\chi_{6,n})=\overline{\leg {-D^2_K}{n}}_6g_K(\chi_{6,n}).
\end{align*}
This leads to the case $j=6$ of \eqref{tauprim1}. \newline

  Lastly, the computations on \cite[p. 347]{G&Zhao7}
\begin{align*}
  \leg {\overline{n}}{n}_4=\begin{cases}
   \leg {2i}{n}_4 \qquad & \text{if $\leg {-1}{n}_4=1$}, \\ \\
     i\leg {2i}{n}_4 \qquad & \text{if $\leg {-1}{n}_4=-1$ }.
    \end{cases}
\end{align*}
  We conclude from the above and \eqref{tau} that the identity for $j=4$ in \eqref{tauprim1} is valid. This completes the proof of the lemma.
\end{proof}

\subsection{Estimating $L$-functions}
\label{sect: Lfcn}

  Let $K$ be an imaginary quadratic  number field of class number one and $\psi$ a Hecke character modulo $m$ of trivial infinite type. We reserve the letter $\varpi$ for a prime in $\mathcal O_K$. Write $\chi$ for the primitive Hecke character modulo $q$ that induces $\psi$.  Moreover, write $m=m_1m_2$ uniquely such that $(m_1, q)=1$ and that $\varpi |m_2 \Rightarrow \varpi|q$. Then
\begin{align}
\label{Ldecomp}
\begin{split}
	L(s,  \psi )=&	\prod_{\varpi | m_1}(1-\chi(\varpi)N(\varpi)^{-s}) \cdot
L(s, \chi).
\end{split}
\end{align}
   Note that
\begin{align*}
 \Big |1-\chi(\varpi)N(\varpi)^{-s}\Big | \leq 2N(\varpi)^{\max (0,-\Re(s))}.
\end{align*}
  It follows that
\begin{align}
\label{Lnbound}
\begin{split}
 \Big | \prod_{\varpi | m_1}\Big(1-\chi(\varpi)N(\varpi)^{-s} \Big )\Big | \ll
2^{\mathcal{W}_K(q_1)}N(m_1)^{\max (0,-\Re(s))} \ll N(m_1)^{\max (0,-\Re(s))+\varepsilon},
\end{split}
\end{align}
  where $\mathcal{W}_K$ denotes the number of distinct prime factors of $n$ and the last estimation above follows from the well-known bound
(which can be derived similar to the proof of the classical case over $\mq$ given in\cite[Theorem 2.10]{MVa1})
\begin{align*}
   \mathcal{W}_K(h) \ll \frac {\log N(h)}{\log \log N(h)}, \quad \mbox{for} \quad N(h) \geq 3.
\end{align*}

  We conclude from \eqref{Ldecomp} and \eqref{Lnbound} that
\begin{align}
\label{Lnbound10}
\begin{split}
	L(s,  \psi ) \ll &	N(m)^{\max (0,-\Re(s))+\varepsilon}|L(s, \chi)|.
\end{split}
\end{align}

Next, the Lindel\"of hypothesis asserts that when $\Re(s)=1/2$,
\begin{align}
\label{LH}
\begin{split}
  L(s,  \chi) \ll (N(q)(1+|s|))^{\varepsilon}.
\end{split}
\end{align}

  Now, a well-known result of E. Hecke shows that $L(s, \chi)$ has an
analytic continuation to the whole complex plane and satisfies the
functional equation (see \cite[Theorem 3.8]{iwakow}). In particular, $(s-1)L(s, \chi)$ is homomorphic for $s \geq 1/2$.
As $L(s, \chi) \ll 1$ for $s \geq 3/2$, we apply the Phragmen-Lindel\"of principle (see \cite[Theorem 5.33]{iwakow}) to deduce from this and \eqref{LH} that for all $\Re(s) \geq 1/2$,
\begin{align*}
\begin{split}
  (s-1)L(s,  \chi) \ll N(q)^{\varepsilon}((1+|s|))^{1+\varepsilon}.
\end{split}
\end{align*}
  The above, together with \eqref{Lnbound10} now implies that for all $\Re(s) \geq 1/2$, we have, under the Lindel\"of hypothesis,
\begin{align}
\label{LHgen}
\begin{split}
  (s-1)L(s,  \psi) \ll N(m)^{\varepsilon}(1+|s|)^{1+\varepsilon}.
\end{split}
\end{align}

  Our discussions above hold true for other $L$-functions as well.  In particular, we note that the functional equation for Dirichlet $L$-functions over $\mq$ given in \cite[\S 9]{Da} that for any primitive Dirichlet character $\chi$ modulo $q$,
\begin{align}
\label{fneqnquad}
  L(s, \chi) = \frac {\tau(\chi)}{i^{\af}q^{1/2}}\Big( \frac {q} {\pi} \Big)^{1/2-s}\frac {\Gamma(\frac {1-s+\af}{2})}{\Gamma (\frac {s+\af}2)}L(1-s, \overline \chi),
\end{align}
  where $\af=0$ or $1$ be given by $\chi(-1)=(-1)^{\af}$. \newline

Now, Stirling's formula (\cite[(5.113)]{iwakow}) gives that, for constants $a_0$, $b_0$,
\begin{align}
\label{Stirlingratio}
  \frac {\Gamma(a_0(1-s)+ b_0)}{\Gamma (a_0s+ b_0)} \ll (1+|s|)^{a_0(1-2\Re (s))}.
\end{align}

 We deduce from \eqref{fneqnquad} and \eqref{Stirlingratio} that
\begin{align}
\label{Ldecompinv2}
  L(s,  \chi) \ll &  (q(1+|s|))^{1/2-\Re(s)+\varepsilon}|L( 1-s, \overline{\chi})|,  \quad \Re(s)<1/2.
\end{align}

 Further, the convexity bound for $L(s, \widehat \chi)$ (see \cite[Exercise 3, p. 100]{iwakow}) asserts that if $\chi$ is not principal, then
\begin{align}
\label{Lconvexbound}
\begin{split}
 L( s,  \chi) \ll
\begin{cases}
 &  \left (q(1+|s|) \right)^{(1-\Re(s))/2+\varepsilon}, \quad 0 \leq \Re(s) \leq 1, \\
 &  1, \quad \Re(s)>1.
\end{cases}
\end{split}
\end{align}

  This together with \eqref{Ldecompinv2} then implies that when $\Re(s)<1/2$
\begin{align}
\label{fneqnquad1}
  L(s, \chi) \ll (q(1+|s|))^{1/2-\Re(s)+\varepsilon}|L(1-s, \overline{\chi})|.
\end{align}

  We conclude from \eqref{Lconvexbound}-\eqref{fneqnquad1} that
\begin{align}
\label{Lchidbound}
\begin{split}
   L(s, \chi) \ll \begin{cases}
   1 \qquad & \Re(s) >1,\\
   (q(1+|s|))^{(1-\Re(s))/2+\varepsilon} \qquad & 0\leq \Re(s) <1,\\
    (q(1+|s|))^{1/2-\Re(s)+\varepsilon} \qquad & \Re(s) < 0.
\end{cases}
\end{split}
\end{align}

\subsection{Analytic behavior of Dirichlet series associated with Gauss sums}
   For $K=\mq(i)$ or $\mq(\sqrt{-3})$ and let $j in \natn$ and the $j$-th order residue symbol is defined over $K$.
For any Hecke character $\chi$ of trivial infinite type, the norm of whose modulus is bounded and that $\chi(m) = 0$ at elements $m \in
\mathcal O_K$ whenever $\chi_{j,m}$ is undefined, we let
\begin{align*}
   h_j(r,s;\chi)=\sum_{\substack{ m \in \mathcal O_K \\ m \text{ primary} \\ (m,r)=1}}\frac {\chi(m)g_{K}(r,\chi_{j,m})}{N(m)^s}.
\end{align*}

    The following lemma gives the analytic behavior of $h_j(r,s;\chi)$.
 \begin{lemma}
\label{lem1} With the notation as above and $j \geq 3$, the function $h_j(r,s;\chi)$ has meromorphic continuation to $\comc$.  It
is holomorphic in the
region $\sigma=\Re(s) > 1$ except possibly for a simple pole at $s = 1+\frac 1j$. For any $\varepsilon>0$, letting $\sigma_1 = 3/2+\varepsilon$,
then for $\sigma_1
\geq \sigma \geq \sigma_1-1/2$, we have
\begin{align}
\label{hbound}
  |((j(s-1))^2-1)h_j(r,s;\chi)| \ll N(r)^{(\sigma_1-\sigma+\varepsilon)/2}(1+|s|^2)^{ (j-1)/2 \cdot (\sigma_1-\sigma+\varepsilon)}.
\end{align}
  For $\sigma >\sigma_1$, we have
\begin{equation*}
  |h_j(r,s;\chi)| \ll 1.
\end{equation*}
\end{lemma}

We note that Lemma~\ref{lem1} is essentially contained in the Lemma give on page 200 of \cite{P}, except that the estimate in \eqref{hbound} is stated only for $|s-(1+1/j)|> 1/(2j)$. An inspection of the proof there  (see the second display on page 205 of \cite{P}) shows that \eqref{hbound} is in fact valid in the region of our interest.

\subsection{Some results on multivariable complex functions}
	
   We gather here some results from multivariable complex analysis. We begin with the notation of a tube domain.
\begin{defin}
		An open set $T\subset\mc^n$ is a tube if there is an open set $U\subset\mr^n$ such that $T=\{z\in\mc^n:\ \Re(z)\in U\}.$
\end{defin}
	
   For a set $U\subset\mr^n$, we define $T(U)=U+i\mr^n\subset \mc^n$.  We quote the following Bochner's Tube Theorem \cite{Boc}.
\begin{theorem}
\label{Bochner}
		Let $U\subset\mr^n$ be a connected open set and $f(z)$ a function holomorphic on $T(U)$. Then $f(z)$ has a holomorphic continuation to
the convex hull of $T(U)$.
\end{theorem}

 We denote the convex hull of an open set $T\subset\mc^n$ by $\widehat T$.  Our next result is \cite[Proposition C.5]{Cech1} and amounts to asserting that holomorphic continuations to $\widehat T$ inherit the bounds of the original functions on $T$.
\begin{prop}
\label{Extending inequalities}
		Assume that $T\subset \mc^n$ is a tube domain, $g,h:T\rightarrow \mc$ are holomorphic functions, and let $\tilde g,\tilde h$ be their
holomorphic continuations to $\widehat T$. If  $|g(z)|\leq |h(z)|$ for all $z\in T$ and $h(z)$ is nonzero in $T$, then also $|\tilde g(z)|\leq
|\tilde h(z)|$ for all $z\in \widehat T$.
\end{prop}

\section{Proof of Theorem \ref{firstmoment} }
\label{sec Poisson}

\subsection{Initial Treatment}

  We define functions $A_j(s,w)$ for $j=2,3, 4$ or $6$ with $\Re(s)$, $\Re(w)$ large enough by the absolutely convergent double Dirichlet
series
\begin{align}
\label{Aswexp}
\begin{split}
A_j(s,w)=& \sum_{(q,j)=1}\ \sumstar_{\substack{\chi \bmod {q} \\ \chi^j=\chi_0}}\frac{L(w, \chi)}{q^s}= \sumprime_{n \odd }\frac{L(w,  \widehat \chi_{j,n})}{N(n)^s}.
\end{split}
\end{align}
Here $\sum'$ indicates the sum runs over square-free elements of $\mathcal O_K$ without any rational prime divisor.  The last equality above follows from Lemma \ref{lemma:quarticclass}. Here and after, we adapt the convention that when a sextic character is involved in a sum, by primary we mean $E$-primary. \newline

  We apply \eqref{Aswexp} and the Mellin inversion to get that for $\Re(s)=c$ sufficiently large,
\begin{align}
\label{MellinInversionj}
\begin{split}
\sum_{(q,j)=1}\ \sumstar_{\substack{\chi \pmod {q} \\ \chi^j=\chi_0}}L(w, \chi)\Phi \left( \frac
{q}Q \right)=& \frac1{2\pi i}\int\limits_{(c)}A_j(s,w)Q^s\widehat \Phi(s) \dif s.
\end{split}
\end{align}
Recall that the Mellin transform $\widehat{f}$ for a function $f$ is defined to be
\begin{align*} 
     \widehat{f}(s) =\int\limits^{\infty}_0f(t)t^s\frac {\dif t}{t}.
\end{align*}

The next few sections are devoted to study the analytical properties of $A_j(s,w)$.

\subsection{First region of absolute convergence of $A_j(s,w)$}
\label{Sec: first region}

Let $\mu_K$ denote the M\"obius function defined on any number field $K$ and we note that in particular $\mu_{\mq}(d)=\mu(|d|)$ for any $d\in \mz$, where $\mu$ is the usual M\"obius function. We apply the M\"obius inversion to remove the condition that $n$ has no rational prime divisor.  This leads to
\begin{align} \label{Aswexp1}
\begin{split}
A_j(s,w)=& \sum^{\infty}_{m=1}\sumprime_{\substack{0 \neq n \in \mathcal O_K \\ n \odd}}\frac{\leg
{m}n_j}{N(n)^sm^w} = \sum^{\infty}_{m=1} \frac 1{m^w}\sum_{\substack{d \in \mz \\ d \odd }} \frac {\mu_{\mq}(d)\chi^{(m)}_j(d)}{d^{2s}}\sum_{\substack{0 \neq n \in \mathcal O_K \\ n \odd \\
dn \text{ square-free} \ \in \mathcal O_K}}\frac{\chi^{(m)}_j(n)}{N(n)^s}.
\end{split}
\end{align}

As $d$ is square-free in $\mz$, it is also square-free as an element of $\mathcal O_K$ so that the condition that $dn$ is square-free is equivalent to $n$ square-free in $\mathcal O_K$ and $(d,n) = 1$.  Thus we obtain that the inner-most sum over $n$ in \eqref{Aswexp1} is
\begin{align}
\label{sumovern}
\begin{split}
\sum_{\substack{0 \neq n \in \mathcal O_K \\ n \odd, (n,d)=1\\
n \text{ square-free} \ \in \mathcal O_K}}\frac{\chi^{(m)}_j(n)}{N(n)^s}=\prod_{(\varpi, jd)=1}\Big(1+\frac{\chi^{(m)}_j(\varpi)}{N(\varpi)^s}\Big )=\frac {L(s, \chi^{(m)}_j)}{L(2s, \chi^{(m)}_j)}\prod_{\varpi| jd}\Big(1+\frac{\chi^{(m)}_j(\varpi)}{N(\varpi)^s}\Big )^{-1}.
\end{split}
\end{align}

  Note that it follows from \eqref{chivalue} that we have $\chi^{(m)}_j(d)=1$ for $d \in \mz, (d,m)=1$. This implies that
\begin{align}
\label{sumd}
\begin{split}
\sum_{\substack{d \in \mz \\ d \odd }} \frac {\mu_{\mq}(d)\chi^{(m)}_j(d)}{d^{2s}}\prod_{\varpi| d}\Big(1+\frac{\chi^{(m)}_j(\varpi)}{N(\varpi)^s}\Big )^{-1} & =\sum_{\substack{d \in \mz \\ d \odd \\ (d,m)=1 }} \frac {\mu_{\mq}(d)}{d^{2s}}\prod_{\varpi| d}\Big(1+\frac{\chi^{(m)}_j(\varpi)}{N(\varpi)^s}\Big )^{-1} \\
& =:  P(s, \chi^{(m)}_j)\prod_{p| \frac {m}{(m,j)}}\Big(1-\frac {1}{p^{2s}}\prod_{\varpi| p}\Big(1+\frac{\chi^{(m)}_j(\varpi)}{N(\varpi)^s}\Big )^{-1}\Big )^{-1},
\end{split}
\end{align}
  where for any Hecke character $\chi$ of trivial infinite type, we define
\begin{align}
\label{Pdef}
\begin{split}
 P(s, \chi)=\prod_{(p,j)=1}\Big(1-\frac {1}{p^{2s}}\prod_{\varpi| p}\Big(1+\frac{\chi(\varpi)}{N(\varpi)^s}\Big )^{-1}\Big ).
\end{split}
\end{align}

   It is easy to see that for $\Re(s)>0$,
\begin{align*}
\begin{split}
 P(s, \chi^{(m)}_j)=\prod_{p}\Big(1-\frac {1}{p^{2s}}+O \Big( \frac 1{p^{3s}} \Big)\Big )=\zeta(2s)^{-1}\prod_{p}(1+O \Big(\frac 1{p^{3s}} \Big)\Big ).
\end{split}
\end{align*}

  It follows that $P(s, \chi^{(m)}_j)$ converges absolutely for $\Re(s)>1/2$. Moreover, we deduce from \eqref{Aswexp1}--\eqref{sumd} that if $\Re(s)>1/2$,
\begin{align}
\label{Ajswexp2}
\begin{split}
 A_j(s,w)=& \sum^{\infty}_{m=1}  \frac{P(s, \chi^{(m)}_j)}{m^w}\frac {L(s, \chi^{(m)}_j)}{L(2s, \chi^{(m)}_j)}\prod_{\varpi| j}\Big(1+\frac{\chi^{(m)}_j(\varpi)}{N(\varpi)^s}\Big )^{-1}\prod_{p|\frac {m}{(m,j)}}\Big( 1-\frac {1}{p^{2s}}\prod_{\varpi| p}\Big(1+\frac{\chi^{(m)}_j(\varpi)}{N(\varpi)^s}\Big )^{-1} \Big)^{-1}.
\end{split}
\end{align}

Furthermore, from \eqref{Ajswexp2} and bounds similar to that given in \eqref{Lnbound}, for $\Re(s)>1/2$, we have
\begin{align}
\label{Aswboundsumm}
\begin{split}
 (s-1)A_j(s,w) \ll & \sum_{\substack{m}}\frac{|(s-1)L(s, \chi^{(m)}_j)|}{m^{\Re (w)-\varepsilon}},
\end{split}
\end{align}
for any $\varepsilon >0$. \newline

 Now, \eqref{LHgen}, together with the Lindel\"of hypothesis, we get that, save for a simple pole at $s=1$, the sum on the right-hand side of \eqref{Aswboundsumm} converges in the region
\begin{equation*}
		S_{1}=\{(s,w): \ \Re(s) > 1/2, \ \Re(w)>1 \}.
\end{equation*}

   In what follows, we shall define similar regions $S_{i, j}$, $S_j$ and $R$. For these regions, we use the convention that for any real number $\delta$,
\begin{align*}
\begin{split}
 S_{i, j,\delta}=: \{ (s,w)+\delta (1,1) : (s,w) \in S_{i,j} \}, \quad S_{j,\delta}=: \{ (s,w)+\delta (1,1) : (s,w) \in S_j \}, \quad R_{\delta}=: \{ (s,w)+\delta (1,1) : (s,w) \in R \}.
\end{split}
\end{align*}

Using this notation, \eqref{LHgen} renders that in $S_{1, \varepsilon}$,
\begin{align}
\label{Aswboundwlarge}
\begin{split}
 |(s-1)A_j(s,w)| \ll & (1+|s|)^{1+\varepsilon}.
\end{split}
\end{align}

\subsection{Second region of absolute convergence of $A_j(s,w)$}

  We apply the convexity bound \eqref{Lchidbound} to $L(s, \chi)$ in $A_j(s,w)$ given by the first equality of \eqref{Aswexp}. We infer that $A_j(s,w)$ is holomorphic in the region
\begin{equation*}
		S_{2,1}=\{(s,w): \ \Re(s)>1,  \ \Re(s+ w/2)> 3/2, \ \Re(w) \geq 1/2 \}.
\end{equation*}

  Moreover, in the region $S_{2,1, \varepsilon}$, we have
\begin{align}
\label{Aswboundslarge}
\begin{split}
 |A_j(s,w)| \ll &  (1+|w|)^{\max \{(1-\Re(w))/2, 0\}+\varepsilon}.
\end{split}
\end{align}

   Next, applying \eqref{Ldecompinv2}, as well as the remark in the paragraph below, to $L(w, \chi)$ enables us to deduce from \eqref{Aswexp} that for $\Re(w) < 1/2$,
\begin{align}
\label{Aswzexpdbound}
 A_j(s,w) \ll  &  (1+|w|)^{1/2-\Re(w)}\sum_{(q,j)=1}\ \sumstar_{\substack{\chi \bmod {q} \\ \chi^j=\chi_0}}\frac{|L(1-w, \chi)|}{q^{\Re(s+w)-1/2}}.
\end{align}

   Again by \eqref{Lchidbound}, the sums on the right-hand side of \eqref{Aswzexpdbound} converge in
\begin{equation*}
		S_{2,2}=\{(s,w): \ \Re(s+w)>  3/2, \ \Re(s+w/2)> 3/2, \ \Re(w) \leq 1/2 \},
\end{equation*}
and in $S_{2,2, \varepsilon}$,
\begin{align}
\label{Aswboundslarge1}
\begin{split}
 |A_j(s,w)| \ll & (1+|w|)^{1/2-\Re(w)+\max \{\Re(w)/2, 0\}+\varepsilon}.
\end{split}
\end{align}

  Now, set
\begin{equation*}
		S_2 =: S_{2,1} \cup S_{2,2} =\{(s,w):\ \Re(s+w)> 3/2, \ \Re(s+w/2)> 3/2 \}.
\end{equation*}
   It follows that $A_j(s,w)$ converges absolutely in $S_2$. From \eqref{Aswboundslarge} and \eqref{Aswboundslarge1}, we deduce that in $S_{2, \varepsilon}$,
\begin{align} \label{Aswbound2}
\begin{split}
 |A_j(s,w)| \ll & (1+|w|)^{\max \{1/2-\Re(w), (1-\Re(w))/2, 0 \}+\varepsilon}.
\end{split}
\end{align}

  Lastly, the functional equation \eqref{fneqnquad} gives that for $\Re(w) < 1/2$,
\begin{align*}
 \sum_{(q,j)=1}\ \sumstar_{\substack{\chi \bmod {q} \\ \chi^j=\chi_0}}\frac{L(w, \chi)}{q^s}= \sum_{(q,j)=1}\ \sumstar_{\substack{\chi \bmod {q} \\ \chi^j=\chi_0}} i^{-\af}\pi^{w-1/2}\frac {\Gamma(\frac {1-w+\af}{2})}{\Gamma (\frac {w+\af}2)}\frac{\tau(\chi) L(1-w, \overline \chi)}{q^{s+w}} .
\end{align*}
  We now identify any $\chi$ modulo $q$ in the above expression with $\widehat \chi_{j, n}$ to deduce from Lemma \ref{quarticGausssum1} that
\begin{equation*}
\tau(\widehat\chi_{j,n})= \begin{cases} \overline{\leg {c_K}{n}}_jg_K(\chi_{j,n}) & j=3, 6, \\  i^{(1-\chi_{4,n}(-1))/2}\overline{\leg {c_K}{n}}_jg_K(\chi_{j,n}) &  j=4, \end{cases} \quad \mbox{where} \quad  C_K = \begin{cases}
      \sqrt{D_K}, \qquad & j=3, \\
     (2i)^3, \qquad & j=4, \\
      -D^2_K, \qquad & j=6.
    \end{cases}
\end{equation*}

  We further define for $a=\pm 1$,
\begin{align*}
\begin{split}
 C_{j,a}(s, w)= \sum_{m=1}^\infty\frac{1}{m^{w}} \sumprime_{n \odd }\frac{\overline{\leg{a m c_K}{n}_j} g_K(\chi_{j,n})}{N(n)^{s}}.
\end{split}
\end{align*}

  We then use
\begin{align*}
\begin{split}
 \frac 12(\chi_{j,n}(1)\pm \chi_{j,n}(-1))
\end{split}
\end{align*}
  to detect the condition $\af =\pm 1$.  This enables us to deduce from the above and \eqref{fneqnquad} that
\begin{align}
\label{AswexpC}
\begin{split}
A_j(s,w)= \frac 12\frac{\pi^{w-1/2}\Gamma\bfrac{1-w}{2}}{\Gamma\bfrac {w}2} & \Big (C_{j,1}(s+w, 1-w)+C_{j,-1}(s+w, 1-w)\Big ) \\
&+\frac {i^{-a_j}}2\frac{\pi^{w-1/2}\Gamma\bfrac{2-w}{2}}{\Gamma\bfrac {1+w}2}\Big (C_{j,1}(s+w, 1-w)-C_{j,-1}(s+w, 1-w)\Big ),
\end{split}
\end{align}
  where $a_j=0$ if $j=3, 6$ and $a_j=1$ for $j=4$. \newline

   We apply the M\"obius inversion to remove the condition that $n$ has no rational prime divisor and simplify the resulting expression using the identity \eqref{grel} and \eqref{2.1}.  We get that for $a=\pm 1$,
\begin{align*}
\begin{split}
C_{j,a}(s, w)= \sum_{m=1}^\infty\frac{1}{m^{w}} \sum_{\substack{d \in \mz \\ d \odd }} \frac {\mu_{\mq}(d)g_K( \chi_{j,d})\overline{\leg{a m C_K}{d}_j}}{d^{2s}}H_{j,a}(s, md^{j-2}C_K),
\end{split}
\end{align*}
  where any $l \in \mathcal O_K$,
\begin{align*}
\begin{split}
 H_{j,a}(s, l)= \sum_{n \odd }\frac{\overline{\leg{al}{n}_j} g_K(\chi_{j,n})}{N(n)^{s}}= \sum_{\substack{n \odd \\ (n, l)=1} }\frac{\overline{\leg{al}{n}_j} g_K(\chi_{j,n})}{N(n)^{s}}=\sum_{\substack{n \odd \\ (n, al)=1} }\frac{ g_K(al, \chi_{j,n})}{N(n)^{s}},
\end{split}
\end{align*}
and the last equality above follows from \eqref{eq:gmult}. \newline

 We bound $g_K(\chi_{j,d})$ trivially by $\sqrt{N(d)}=d$ and infer from Lemma \ref{lem1} that except for a simple pole at $s = 1/j$, $C_{j,a}(s,w)$ is holomorphic in the region
\begin{align}
\label{Cj1region0}
\begin{split}
	R =: &	\{(s,w):\ \Re(w)>1, \ \Re(w+ s)> 5/2, \ \Re(s)>1, \ \Re(2s+ (j-2)s)>2+ 3(j-2)/2  \} \\
=& \{(s,w):\ \Re(w)>1, \ \Re(w+ s)> 5/2,  \ \Re(s)> 3/2- 1/j \}.
\end{split}
\end{align}

  Moreover, in $R_{\delta}$ we have
\begin{align}
\label{Cj1bound}
		(s-1-1/j)C_{j,a}(s,w) \ll (1+|s|^2)^{ \max \{ (j-1)/2 \cdot (3/2-\Re(w)), 0\}+\varepsilon)}.
\end{align}

  It follows from \eqref{AswexpC} and \eqref{Cj1region0} that the function $(s+w-1-1/j)A_j(s,w)$ can be extended to
the region
\begin{align*}
		S_{3,j}=& \{(s,w): \ \Re(1-w)>1, \ \Re(s)> 3/2, \ \Re(s+w)>3/2-1/j \}.
\end{align*}

   Also, by \eqref{Stirlingratio}, \eqref{AswexpC} and \eqref{Cj1bound}, we have, in $S_{3,j, \delta}$,
\begin{align} \label{Aj1bound}
\begin{split}
		(s+w-1-1/j)A_{j}(s,w) \ll &  (1+|w|)^{1/2-\Re(w)}(1+|s+w|^2)^{ \max \{ (j-1)/2 \cdot(1/2+\Re(w)), 0\}+\varepsilon)} \\
\ll & (1+|w|)^{1/2-\Re(w)+\max \{(j-1)(1/2+\Re(w)), 0\}+\varepsilon)}(1+|s|)^{ \max \{ (j-1)(1/2+\Re(w)), 0\}+\varepsilon)}.
\end{split}
\end{align}

  One checks directly that the union of $S_1, S_2$ and $S_{3,j}$ is connected and the points $(1/2, 1)$, $(3/2, -1/j)$ are on the boundary of this union.  The convex hull of $S_1, S_2$ and $S_{3,j}$ equals
\begin{equation*}
		S_{4,j}=\{(s,w): \Re(s)>1/2, \ \Re((1+1/j)s+w)> 3/2+ 1/(2j), \ \Re(s+w)> 3/2- 1/j \}.
\end{equation*}
 We then conclude from Theorem \ref{Bochner} that $(s-1)(s+w-1-1/j)A_j(s,w)$ converges absolutely in the region $S_{4,j}$. Moreover, we deduce from  \eqref{Aswboundwlarge}, \eqref{Aswbound2}, \eqref{Aj1bound} and Proposition \ref{Extending inequalities} that in the region
\begin{equation*}
		S_{4,j} \bigcap \{(s,w): \Re(s)>1/2, \ -1/2< \Re(w)< 5/2 \},
\end{equation*}
  we have
\begin{align}
\label{Aj}
\begin{split}
		(s-1)(s+w-1-1/j)A_{j}(s,w) \ll &  (1+|s|)^{ 3(j-1)+1+\varepsilon}(1+|w|)^{ 3(j-1)+\varepsilon}.
\end{split}
\end{align}

\subsection{Residue of $A_j(s,w)$ at $s=1$}
\label{sec:resA}
	
	We see from \eqref{Ajswexp2}that $A_j(s,w)$ has a pole at $s=1$ arising from the terms with $m= \text{$j$-th power}$.  Let $r_K$ denote the residue of $\zeta_K(s)$ at $s=1$.  We compute directly from \eqref{Aswexp1} and \eqref{sumovern}, upon making a change of variable $m \mapsto m^j$, that
\begin{align*}
 \res_{s=1}A_j(s,w)= \frac {r_K}{\zeta_{K}(2)}\sum^{\infty}_{m=1} \frac 1{m^{jw}}\sum_{\substack{d \in \mz \\ d \odd, (d,m)=1 }} \frac {\mu_{\mq}(d)}{d^{2}}\prod_{\varpi| mjd}\Big(1+\frac{1}{N(\varpi)}\Big )^{-1}  = \frac {r_K}{\zeta_{K}(2)}P(1, \psi_0)Z_j(w),
\end{align*}
 where $\psi_0$ is the principal Hecke character of trivial infinite type modulo $1$, $P$ is defined in \eqref{Pdef} and
\begin{align*}
\begin{split}
 Z_j(w)=\sum^{\infty}_{m=1} \frac 1{m^{jw}}\prod_{\varpi| jm}\Big(1+\frac{1}{N(\varpi)}\Big )^{-1}\prod_{p|\frac {m}{(m,j)}}\Big(1-\frac {1}{p^{2}}\prod_{\varpi| p}\Big(1+\frac{1}{N(\varpi)}\Big )^{-1}\Big )^{-1}.
\end{split}
\end{align*}
  We note here that it is easy to see that $Z(u)$ is holomorphic and bounded for $\Re(u) \geq 1 +\delta>1$. \newline

  We now define constants $C_j$ for $j=3$, $4$ or $6$ by
\begin{align}
\label{eq:c}
\begin{split}
 C_j=\frac {r_K}{\zeta_{K}(2)}P(1, \psi_0)Z_j(\tfrac{1}{2}+\alpha).
\end{split}
\end{align}

  It follows that
\begin{align}
\label{Residue at s=1}
 \res_{s=1}& A_j(s, \tfrac{1}{2}+\alpha) = C_j.
\end{align}

\subsection{Completion of the Proof}

We now evaluate the integrals in \eqref{MellinInversionj} by shifting the line of integration there to $\Re(s)=(2j+1-2j\alpha)/(2j+2)+\varepsilon$ for all $j$ under our consideration.  Note that integration by parts yiedls that for any rational integer $E \geq 0$,
\begin{align}
\label{whatbound}
 \widehat \Phi(s)  \ll  \frac{1}{(1+|s|)^{E}}.
\end{align}

  The integral on the new line can be absorbed into the $O$-term in \eqref{eq:1} upon using \eqref{Aj} and \eqref{whatbound}.  We also encounter two simple poles at $s=1$ and $s=1/2+1/j-\alpha$ in the shift. The residue at $s=1$ is given in \eqref{Residue at s=1} and direct computation then leads to the main term given in \eqref{eq:1}. The residue at $s=1/2+1/j-\alpha$ leads to an error term of size $O(Q^{1/2+1/j-\alpha})$ by \eqref{Aj} and \eqref{whatbound} again. As this can also be absorbed into the $O$-term in \eqref{eq:1}, this completes the proof of Theorem \ref{firstmoment}.

\vspace*{.5cm}

\noindent{\bf Acknowledgments.}   P. G. is supported in part by NSFC Grant 11871082 and L. Z. by the Faculty Research Grant PS43707 at the University of New South Wales.

\bibliography{biblio}
\bibliographystyle{amsxport}

\end{document}